\documentclass[12pt]{article}

\usepackage{hook-two-row}

\title{
	{Character polynomials for two rows and hook partitions}\\
	{\large University of Western Ontario}
	{}
}

\author{Ahmed Umer Ashraf}
\date{\today}

\begin{document}

\maketitle

\begin{abstract}
Representation theory of the symmetric group $\mathfrak{S}_n$ has a very distinctive combinatorial flavor. The conjugacy classes as well as the irreducible characters are indexed by integer partitions  $\lambda \vdash n$. We introduce class functions on $\mathfrak{S}_n$ that count the number of certain tilings of Young diagrams. The counting interpretation gives a uniform expression of these class functions in the ring of character polynomials, as defined by \cite{murnaghanfirst}. A modern treatment of character polynomials is given in \cite{orellana-zabrocki}. We prove a relation between these combinatorial class functions in the (virtual) character ring. From this relation, we were able to prove Goupil's generating function identity \cite{goupil}, which can then be used to derive Rosas' formula \cite{rosas} for Kronecker coefficients of hook shape partitions and two row partitions.                                     
\end{abstract}

\tableofcontents

\section{Introduction}
    
Representation theory of the symmetric group $\mathfrak{S}_n$ employs a good amount of combinatorics of (integer) partitions of $n$. The irreducible representations of $\mathfrak{S}_n$ are indexed by partitions. One way to generate these irreducible representations is through constructing a vector space $M^\lambda$ generated by equivalence classes of tableaux, called \emph{tabloids} of shape $\lambda$. And then show that each $M^\lambda$ contains an irreducible representation of $\mathfrak{S}_n$ as a subspace. The number of tabloids of shape $\lambda$ can also be viewed as certain tilings of Young diagram of shape $\lambda$. This motivates us to define class functions over $\mathfrak{S}_n$ that count certain tilings we call \emph{brick tilings}. In this section, we review representations theory of $\mathfrak{S}_n$, and define these class functions. We also recall Doubilet's inversion formula, and face numbers of permutohedron which will be of use in the later sections.

    \subsection{Partitions and compositions}
        A \emph{partition} $\lambda$ of a positive integer $n$, denoted as $\lambda \vdash n$, is a weakly decreasing sequence $\lambda = (\lambda_1, \lambda_2, \cdots, \lambda_r )$ of positive integers adding up to $n$. The positive integers $\lambda_i$ are called \emph{parts} of $\lambda$, and the number of parts is called the \emph{length} of $\lambda$, denoted as $\ell(\lambda)$. If we want to emphasize that $\ell(\lambda) = r$, we write $\lambda \vdash_r n$. In relation to $\lambda$, the integer $n$ is called the \emph{weight} of $\lambda$, and it is denoted by $|\lambda|$. It is also useful to write $\lambda = (1^{m_1}, 2^{m_2}, \dots, n^{m_n})$, where $m_i$ denote the multiplicity of $i$ in the partition $\lambda$. Given a partition $\lambda \vdash n$, the \emph{reduced partition} $\left< \lambda \right>$ is a partition of $n-\lambda_1$ defined as $\left< \lambda \right> =(\lambda_2, \dots, \lambda_r) \vdash |\lambda| - \lambda_1$. Similarly, for a partition $\lambda \vdash k$, and a positive integer $n \geq k + \lambda_1$, the \emph{augmented partition} $\lambda[n]$ is a partition of $n$ defined as $\lambda[n] =(n-k, \lambda_1, \cdots, \lambda_{r})$. In other words, the \emph{reduced partition} $\left< \lambda \right>$ is a partition we get by removing the first part of $\lambda$, while the augmented partition is a partition we get by augmenting a suitable first part to $\lambda$. We identify partition $\lambda$ with its Young diagram, which is a finite collection of unit cells arranged in left justified rows with $\lambda_i$ cells in the $i$th row. A Young tableau of shape $\lambda$, is a labeling of the cells of the Young diagram of $\lambda$ with integers $1, 2, \dots, n$, with each number occurring exactly once.

 A composition $\mu$ of a positive integer $n$, denoted as $\mu \vDash n$, is a sequence $(\mu_1, \mu_2, \cdots, \mu_r)$ of positive integers adding to $n$. We extend the definitions and notation introduced above for partitions to compositions. Given a composition $\mu$, we denote by $\tilde{\mu}$ the partition obtained by rearranging the parts of $\mu$ in weakly decreasing order.

For any positive integer $n$, let $\Comp(n)$ denote the set of all compositions of $n$. We define a partial order on $\Comp(n)$ in the following manner: Given two compositions $\nu =(\nu_1,\dots, \nu_s)$ and $\mu=(\mu_1,\dots,\mu_t)$ in $\Comp(n)$, we say that $\mu$ covers $\nu$, and write $\nu \lessdot \mu$ if $\ell(\mu)=\ell(\nu) +  1$ and there exists a unique $j$ such that
\begin{align*}
\mu_i &= \begin{cases}
\nu_i~~~&\text{for}~ i<j \\
\nu_{i} + \nu_{i+1}~~~&\text{for}~i=j \\
\nu_{i+1}~~~&\text{for}~i > j
\end{cases}
\end{align*}
i.e. the covering relations are given by adding adjacent entries. The partial order on $\Comp(n)$ induced by this relation is denoted by $\leq$. For instance, in figure \ref{composet5} we see the Hasse diagram of this partial order on $\Comp(5)$.

\begin{figure}
    \centering
    \begin{tikzpicture}[scale=.4]
  \node (a) at (20, 0) {\ytableausetup{boxsize=0.5em}\ydiagram{1,1,1,1,1}};
 \node at (20, -2) {\tiny{(1, 1, 1, 1, 1)}};
  \node (b) at (12,6) {\ytableausetup{boxsize=0.5em}\ydiagram{2,1, 1,1}};
 \node at (12, 4) {\tiny{(2, 1, 1, 1)}};
  \node (c) at (18,6) {\ytableausetup{boxsize=0.5em}\ydiagram{1,2,1,1}};
 \node at (18, 4) {\tiny{(1, 2, 1, 1)}};
  \node (d) at (24,6)  {\ytableausetup{boxsize=0.5em}\ydiagram{1,1,2,1}};
 \node at (24, 4) {\tiny{(1, 1, 2, 1)}};
  \node (e) at (30,6)  {\ytableausetup{boxsize=0.5em}\ydiagram{1,1,1,2}};
 \node at (30,4) {\tiny{(1, 1, 1, 2)}};
  \node (f) at (30,12) {\ytableausetup{boxsize=0.5em}\ydiagram{1,2,2}};
 \node at (30, 10)  {\tiny{(1,2,2)}};
  \node (g) at (6,12)  {\ytableausetup{boxsize=0.5em}\ydiagram{3,1,1}};
 \node at (6,10) {\tiny{(3,1,1)}};
  \node (h) at (24,12)  {\ytableausetup{boxsize=0.5em}\ydiagram{1,3,1}};
 \node at (24, 10) {\tiny{(1,3,1)}};
  \node (i) at (36,12)  {\ytableausetup{boxsize=0.5em}\ydiagram{1,1,3}};
 \node at (36, 10) {\tiny{(1,1,3)}};
  \node (j) at  (12,12)    {\ytableausetup{boxsize=0.5em}\ydiagram{2,2,1}};
 \node at  (12, 10)  {\tiny{(2,2,1)}};
  \node (k) at  (18,12)  {\ytableausetup{boxsize=0.5em}\ydiagram{2,1,2}};
 \node at  (18, 10) {\tiny{(2,1,2)}};
  \node (l) at (12,18)  {\ytableausetup{boxsize=0.5em}\ydiagram{4,1}};
 \node at (12, 16) {\tiny{(4,1)}};
  \node (m) at  (30,18) {\ytableausetup{boxsize=0.5em}\ydiagram{1,4}};
 \node at (30, 16) {\tiny{(1, 4)}};
  \node (o) at (18,18)    {\ytableausetup{boxsize=0.5em}\ydiagram{3,2}};
 \node at (18, 16) {\tiny{(3,2)}};
  \node (n) at   (24,18) {\ytableausetup{boxsize=0.5em}\ydiagram{2,3}};
 \node at (24, 16)  {\tiny{(2,3)}};
  \node (p) at (20,24)  {\ytableausetup{boxsize=0.5em}\ydiagram{5}};
 \node at (20, 22) {\tiny{(5)}};
   \draw (a)  -- (b); \draw  (a) -- (c) ;\draw  (a) -- (d) ; \draw  (a) -- (e);
  \draw (b)  -- (g); \draw  (b) -- (j) ;\draw  (b) -- (k) ;
  \draw (c)  -- (f); \draw (c)  -- (g); \draw (c)  -- (h); 
  \draw (d)  -- (h); \draw (d)  -- (i); \draw (d)  -- (j);
 \draw  (e) -- (f) ;\draw  (e) -- (k) ;\draw  (e) -- (i) ;
\draw  (f) -- (m) ;\draw  (f) -- (o) ;
  \draw (g)  -- (l); \draw (g)  -- (o);
   \draw (h)  -- (l); \draw (h)  -- (m);
   \draw (i)  -- (m);  \draw (i)  -- (n);
   \draw (j)  -- (l); \draw (j)  -- (n); 
   \draw (k)  -- (n); \draw (k)  -- (o); 
   \draw (l)  -- (p); \draw (m)  -- (p); \draw (n)  -- (p); \draw(o) --(p);
\ytableausetup{boxsize=1em}
\end{tikzpicture}
    \caption{Poset of composition $\Comp(5)$}
    \label{composet5}
\end{figure}
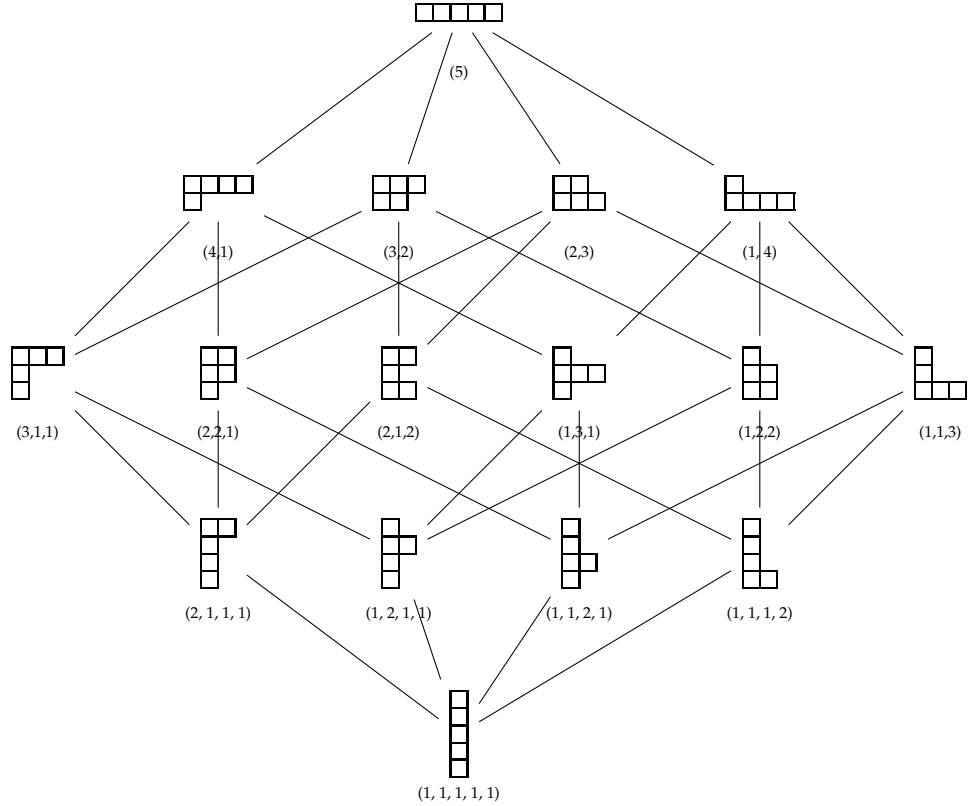

%
%

We identify the abstract group $\mathfrak{S}_n$ with the group of permutations on the set $[n] = \{1, 2, \dots, n\}$. Let $w$ be an element of $\mathfrak{S}_n$, then $w$ can be written as a product of pairwise disjoint cycles, called \emph{cyclic factors} of $w$. Let $r$ denote the number of these cyclic factors including the fixed points (1-cycles). Let $l_i$ be their lengths for $i = 1, \dots, r$. By choosing an element $j_i$ for the $i$-th cyclic factor, we can write 
\begin{align*}
w &= \prod_{i=1}^{r} (j_i, w j_i, \cdots, w^{l_i -1} j_i)
\end{align*}
We can make this notation unique by choosing $j_i$ such that for all positive integer $m$,
\begin{align*}
j_i \geq w^m j_i
\end{align*}
and for all $i = 1, 2, \dots, r-1$, take
\begin{align*}
j_i < j_{i+1}
\end{align*}
Such a unique decomposition is called \emph{canonical cycle decomposition} of $w$. This plays an important role in Foata's first fundamental bijection \cite{foata-schutzenberger}. Note that $l = (l_1, l_2, \cdots, l_r)$ is a composition of $n$. The underlying partition $\tilde{l}$ is called the \emph{cycle type} of $w$ and is denoted as $\cyc(w)$. We know that two permutations $u, w$ belong to the same conjugacy class in $\mathfrak{S}_n$ if and only if $\cyc(u) = \cyc(w)$. 
\begin{example}
If $w = 947213865 \in \mathfrak{S}_9 $, then we have the canonical cycle decomposition 
\begin{align*}
w &=(4,2)(8, 6, 3, 7)  (9, 5,1)
\end{align*}
with $\cyc(w) = (4,3,2) \vdash 9$.
\end{example}

    \subsection{Representation theory of symmetric group $\mathfrak{S}_n$}
        
A tabloid $[t]$ of shape $\lambda$ is an equivalence class of Young tableaux of shape $\lambda$, where we consider two tableaux $t$ and $t'$ equivalent if the entries in each row of $t$ agrees with the corresponding entries in row of $t'$.  
Given the set $T(\lambda)$ of all Young tableaux of shape $\lambda$, there is a natural action of $\mathfrak{S}_n$ on $T(\lambda)$ by just permuting the labels of tableaux.  This induces an action on tabloids. Given a Young tableau $t$, the polytabloid $e_t$ associated to $t$ is defined as the linear combination
\begin{align*}
e_t &:= \sum_{\pi \in C_t} \sgn(\pi) \pi[t] 
\end{align*}
where $C_t$ is the column group associated to $t$, i.e. the subgroup of $\mathfrak{S}_n$ consisting of permutations that only permute elements within each column of $t$. For each partition $\lambda \vdash n$, $\mathbb{C}$-linear combination of polytabloids of shape $\lambda$ gives an irreducible representation of $\mathfrak{S}_n$ over $\mathbb{C}$. This is referred as Specht module $S^\lambda$ corresponding to $\lambda \vdash n$ in the literature \cite{steinberg}. Let $\Cl(\mathfrak{S}_n)$ denote the vector space of class functions on the group $\mathfrak{S}_n$ over $\mathbb{C}$. The characters of Specht modules, $(\chi^\lambda)_{\lambda \vdash n}$, gives a basis for $\Cl(\mathfrak{S}_n)$. There is a scalar product $\left< \cdot , \cdot \right>_{\mathfrak{S}_n}$ on $\Cl(\mathfrak{S}_n)$ defined as
\begin{align*}
\left< \chi^\lambda, \chi^\mu \right> &= \frac{1}{n!} \sum_{\sigma \in \mathfrak{S}_n}  \chi^\lambda(\sigma) \chi^\mu(\sigma)
\end{align*}
and extended linearly. The decomposition of the permutation character in terms of the irreducible character basis $\{\chi^\mu\}$ is given by Young's rule, which gives:
\begin{equation} \label{kostka}
\zeta^\lambda = \sum_{\mu \unlhd \lambda} K_{\mu \lambda} \chi^\mu
\end{equation}
where $K_{\mu\lambda}$, are the Kostka numbers, and the sum is over all partitions $\mu$ which are less than or equal to $\lambda$ in the \emph{dominance order}.
    \subsection{Brick tilings}

A \emph{brick} of length $j$ is a labelled horizontal array of $j$ unit cells. We will view it as a $1 \times j$ rectangle. To each $j$-cycle $a = (a_1 a_2 \cdots a_j)$ in the canonical cycle decomposition of $w$, we can associate a brick of length $j$ with $i$th square labelled $a_i$.  Given $w \in \mathfrak{S}_n$ where $w$ has cycle type $\lambda = (\lambda_1, \cdots, \lambda_r) \vdash n$, we denote by $B_w$ the set of associated bricks of length $\lambda_1, \cdots, \lambda_r$ corresponding to each cyclic factor in the canonical cycle decomposition of $w$. Note that the $1$-cycles correspond to $1 \times 1$ square bricks. A \emph{tiling} of a diagram $\lambda \vdash n$ by a set of bricks $B$ is a covering of the diagram $\lambda$ with bricks from $B$ such that no brick is used twice and each cell of $\lambda$ is covered by some brick from $B$. 
An \emph{ordered brick tiling} of $\lambda \vdash k$ (or $\lambda \vDash k$) by $w \in \mathfrak{S}_n$ is a tiling of Young diagram of $\lambda$ by bricks from $B_w$, where no brick is in more than one row and the order of the bricks in a row is irrelevant. To be more precise, the ordered brick tiling of shape $\lambda = (\lambda_1, \cdots, \lambda_{k})$ with $w \in \mathfrak{S}_n$ is an ordered tuple $(S_1, S_2, \cdots, S_r)$ of disjoint subsets of $B_w$ such that
\begin{align*}
    \bigg| \bigcup_{U \in S_j} U \bigg| = \lambda_j
\end{align*}
for all $j=1, \cdots, r$. The set of brick tilings of $\lambda \vdash n$ by $w \in \mathfrak{S}_n$ is denoted by $B_w(\lambda)$. So in each element $(S_1, S_2, \dots, S_r)$ of $B_w(\lambda)$ with $\ell(\lambda) = r$, the set $S_i$ represent the set of bricks used to tile the $i$-th row of Young diagram of $\lambda$. Notice since the order of tiles in a row does not matter, therefore there is no ambiguity in this notation. If need be, we write $T_i$ explicitly using the canonical cycle decomposition.

We can define an equivalence relation among brick tilings of shape $\lambda$ as follows: Two brick tilings $S = (S_1, S_2, \dots, S_r)$ and $T = (T_1, T_2, \dots, T_r)$ of  partition $\lambda \vdash_r n$, are equivalent if one is a permutation of other i.e.
\begin{align*}
    (S_1, \cdots, S_r) &= (T_{\pi(1)}, \cdots, T_{\pi(r)})~~\text{for some}~\pi \in \mathfrak{S}_r
\end{align*}
We refer to these equivalence classes of tilings as \emph{unordered brick tilings} of $\lambda$ by $w$, and we denote the set of these equivalence classes by $\widetilde{B}_{w}(\lambda)$. We say a brick tiling is \emph{crackless} whenever we have exactly one tile in each row. Otherwise, we say it is \emph{cracked}. A \emph{crack} in a brick tiling is the occurrence of two tiles in one row of a Young diagram. If a row contains $c$ many tiles, we say it has $c-1$ cracks, and the number of cracks in a brick tiling is sum of number of cracks in its rows. For a $T$ in $\widetilde{B}_w(\lambda)$ (resp. $B_w(\lambda)$), we call $\lambda$ the \emph{shape of $T$} and denote it by $\sh(T)$. Furthermore, for any subset $A = \{ b_1, b_2, \dots, b_r\}$ of $B_w$, the \emph{shape of $A$}, denoted as $\sh(A)$ is the sequence of lengths of bricks $b_i$ in decreasing order.  

\begin{example} \label{brickexample}
Consider $\lambda = (2, 2, 1) \vdash 5$ and let $u = (3,1)(4)(5,2)$ and $w = (2)(3,1)(4)(5)$ in $\mathfrak{S}_5$. We have $B_u = \left\{\ytableaushort{31}*[*(red!20)]{2} , \ytableaushort{4}*[*(green!20)]{1},  \ytableaushort{52}*[*(blue!20)]{2} \right\}$ and
 $B_w = \left\{\ytableaushort{2}*[*(blue!20)]{1},  \ytableaushort{31}*[*(red!20)]{2} , \ytableaushort{4}*[*(green!20)]{1},  \ytableaushort{5}*[*(yellow!50)]{1} \right\}$. The diagram of $\lambda$ is given by \ydiagram[]{2,2,1}. The ordered and unordered tilings of $\lambda$ by $u$ and $w$ are given below
\begin{align*}
\widetilde{B}_u(\lambda) &= \left\{ \ytableaushort{31,52,4}*[*(red!20)]{2} *[*(blue!20)]{0,2}*[*(green!20)]{0,0,1}, \ytableaushort{52,31,4}*[*(blue!20)]{2} *[*(red!20)]{0,2}*[*(green!20)]{0,0,1} \right\} \\
\widetilde{B}_w(\lambda) &= \left\{ \ytableaushort{31,24,5}*[*(red!20)]{2} *[*(blue!20)]{0,1}*[*(green!20)]{0,2}*[*(yellow!50)]{0,0,1} , \ytableaushort{31,25,4}*[*(red!20)]{2} *[*(blue!20)]{0,1}*[*(green!20)]{0,0,1}*[*(yellow!50)]{0,2}, \ytableaushort{31,45,2}*[*(red!20)]{2} *[*(blue!20)]{0,0,1}*[*(green!20)]{0,1}*[*(yellow!50)]{0,2}, 
\ytableaushort{24,31,5}*[*(red!20)]{0,2} *[*(blue!20)]{1}*[*(green!20)]{2}*[*(yellow!50)]{0,0,1} , \ytableaushort{25,31,4}*[*(red!20)]{0,2} *[*(blue!20)]{1}*[*(green!20)]{0,0,1}*[*(yellow!50)]{2}, \ytableaushort{45,31,2}*[*(red!20)]{0,2} *[*(blue!20)]{0,0,1}*[*(green!20)]{1}*[*(yellow!50)]{2}   \right\}
\end{align*}
and
\begin{align*}
{B}_u(\lambda) &= \left\{  \left[ \ytableaushort{31,52,4}*[*(red!20)]{2} *[*(blue!20)]{0,2}*[*(green!20)]{0,0,1} \right] \right\} \\
{B}_w(\lambda) &= \left\{ \left[ \ytableaushort{31,24,5}*[*(red!20)]{2} *[*(blue!20)]{0,1}*[*(green!20)]{0,2}*[*(yellow!50)]{0,0,1} \right], \left[ \ytableaushort{31,25,4}*[*(red!20)]{2} *[*(blue!20)]{0,1}*[*(green!20)]{0,0,1}*[*(yellow!50)]{0,2} \right], \left[ \ytableaushort{31,45,2}*[*(red!20)]{2} *[*(blue!20)]{0,0,1}*[*(green!20)]{0,1}*[*(yellow!50)]{0,2} \right] \right\}.
\end{align*}
Note that all the elements of $\widetilde{B}_u(\lambda)$ are crackless, but all the elements of $\widetilde{B}_w(\lambda)$ are cracked with one crack each. 
\end{example}
    \subsection{Tiling class functions}
        
In this section, we define class functions on $\mathfrak{S}_n$ that count the number of different types of brick tilings. Let $k$ be a positive integer less than or equal to $n$ and let $\lambda \vdash k$ (or $\lambda \vDash k$). We define functions $\zeta^\lambda, \xi^\lambda, \eta^\lambda: \mathfrak{S}_n \longrightarrow \mathbb{N}$ as follow 
\begin{align*}
\zeta^\lambda(w) &:= ~\text{number of ordered brick tilings of $\lambda$ by $w$ i.e. ~}|\widetilde{B}_w(\lambda)| \\
\xi^\lambda(w) &:=  ~\text{number of unordered brick tilings of $\lambda$ by $w$ i.e.~} |{B}_w(\lambda)| \\
\eta^\lambda(w) &:=~\text{number of unordered crackless brick tilings of $\lambda$ by $w$}
\end{align*}
These are class functions on $\mathfrak{S}_n$. What that means is that they are constant on each conjugacy class of $\mathfrak{S}_n$. The notion of tiling of $\lambda \vdash k$ with $w \in \mathfrak{S}_n$ when $k \neq n$ still make sense. For $k > n$, all of the above class functions are identically zero, so we keep the condition $k \leq n$. 

\begin{example} \label{classexample}
Going back to example \ref{brickexample}, for the respective partition $\lambda$ and permutations $u, w$, we have
\begin{table}[h]
    \centering
    \begin{tabular}{cc}
       $ \zeta^\lambda(u) = 2$  & $\zeta^\lambda(w) = 6$ \\
        $\xi^\lambda(u) = 1$ & $  \xi^\lambda(w) = 3$ \\
        $\eta^\lambda(u) = 1 $ & $\eta^\lambda(w) = 0$
    \end{tabular}
    \caption{Combinatorial class functions for example \ref{brickexample}}
    \label{combchar}
\end{table}
\end{example}

The interesting case is when $k=n$. This means that all of the bricks from $B_w$ are utilized to cover diagram $\lambda \vdash n$. So the ordered brick tilings correspond precisely to tabloids, and hence we have the following result.

\begin{theorem}
For $\lambda \vdash n$, 
\begin{itemize}
\item
$\zeta^\lambda$ is the character corresponding to the permutation representation $M^\lambda$ of $\mathfrak{S}_n$. Furthermore,
\begin{align*}
\xi^\lambda &= \frac{1}{\lambda!} \zeta^\lambda 
\end{align*}
where $\lambda! := m_1 ! m_2! \cdots $ for $\lambda = (1^{m_1}, 2^{m_2}, \cdots)$.

\item
$\eta^\lambda$ is the indicator function of cycle structure i.e.
\begin{align*}
\eta^\lambda(w) &= \begin{cases}
1 &\text{if}~~\cyc(w) = \lambda \\
0 &\text{otherwise}
\end{cases}
\end{align*}

\end{itemize}
\end{theorem}

\begin{proof}
Recall that for a partition $\lambda = (\lambda_1, \cdots, \lambda_\ell) \vdash n$ and $w \in \mathfrak{S}_n$ with $\cyc(w) = \mu = (\mu_1, \cdots, \mu_r)$, the character of $\mathfrak{S}_n$ corresponding to $M^\lambda$ is the coefficient of $x_1^{\lambda_1} x_2^{\lambda_2} \cdots x_\ell^{\lambda_\ell}$ in the product
\begin{align*}
    \prod_{i=1}^r (x_1^{\mu_i} + \cdots + x_\ell^{\mu_i})
\end{align*}
which precisely counts the ordered brick tilings of $w$ by $\lambda$. Furthermore, to each unordered brick tiling of $\lambda$ we get $\lambda! = m_1 ! m_2! \cdots$ ordered brick tilings by permuting the tilings in the parts of same size, and vice versa. 
For the second part, a crackless tiling of $\lambda$ is only possible if each part of $\cyc(w)$ fits perfectly in a unique part of $\lambda$. This is just another way of saying that $\cyc(w) = \lambda$.

\end{proof}

\begin{lemma}
Let $k \leq n$, then 
\begin{itemize}
\item for any $\mu \vDash k$ the following holds as  identities of class functions on $\mathfrak{S}_n$.
\begin{align*}
\zeta^\mu &= \zeta^{\tilde{\mu}} ~~,~~  \xi^\mu = \xi^{\tilde{\mu}} ~~,~~  \eta^\mu = \eta^{\tilde{\mu}} 
\end{align*}
\item
for any $\lambda \vdash n$ then the following holds as an identity of class functions on $\mathfrak{S}_n$.
\begin{align*}
\zeta^\lambda &= \zeta^{\left< \lambda \right>} 
\end{align*}
and if $\lambda_1 > \lambda_2$, then  $\xi^\lambda = \xi^{\left< \lambda \right>}$ and  $\eta^\lambda = \eta^{\left< \lambda \right>}$.
\end{itemize}
\end{lemma}
\begin{proof}
The first statement is just a consequence of the fact that the brick tilings of any type does not depend in the relative order of parts of $\lambda$. The second says that if the weight of $\lambda$ and the weight of $\cyc(w)$ are equal, then to determine a tiling of $\lambda$ by $w$, we just need to determine a tiling of $\left< \lambda \right>$ by $w$. Because whatever tiles are not being used will have to fit in $\lambda_1$, and there is only one way of doing that, as the order of tiles in rows of $\lambda$ does not matter.
\end{proof}

\begin{example}
Let $\lambda = (2, 2, 1) \vdash 5$  and $u = (3,1)(4)(5,2)$ and $w = (2)(3,1)(4)(5)$, recall
\begin{align*}
\zeta^{\lambda}(u) &= |\widetilde{B}_u(\lambda)| = 2 \\
\zeta^{\lambda}(w) &= |\widetilde{B}_w(\lambda)| = 6 
\end{align*}
In this case, $\left< \lambda \right> = (2,1) \vdash 3$ diagram of $\lambda$ is given by  \ydiagram[]{2,1}, we have
\begin{align*}
\widetilde{B}_u(\left< \lambda \right> ) &= \left\{  \ytableaushort{52,4}*[*(blue!20)]{2}*[*(green!20)]{0,1}, \ytableaushort{31,4}*[*(red!20)]{2}*[*(green!20)]{0,1} \right\} \\
\widetilde{B}_w(\left< \lambda \right>) &= \left\{ \ytableaushort{24,5}*[*(blue!20)]{1}*[*(green!20)]{2}*[*(yellow!50)]{0,1} , \ytableaushort{25,4}*[*(blue!20)]{1}*[*(green!20)]{0,1}*[*(yellow!50)]{2}, \ytableaushort{45,2}*[*(blue!20)]{0,1}*[*(green!20)]{1}*[*(yellow!50)]{2}, 
\ytableaushort{31,5}*[*(red!20)]{2}*[*(yellow!50)]{0,1} , \ytableaushort{31,4}*[*(red!20)]{2} *[*(green!20)]{0,1}, \ytableaushort{31,2}*[*(red!20)]{2} *[*(blue!20)]{0,1}   \right\}
\end{align*}
which is in accordance with table \ref{combchar}.
\end{example}
For $T \subseteq B_w$ be a set of bricks coming from $w \in \mathfrak{S}_n$ and we denote by $\xi^\lambda_T$ the number of tilings of $\lambda \vdash n$ with brick set given by $T$. Then counting \emph{brickwise}, we have
\begin{equation} \label{brickwise}
  \xi^\lambda = \sum_{T \subseteq B_w} \xi^\lambda_T 
\end{equation}
and
\begin{equation} \label{brickwisesh}
\eta^\lambda = \sum_{\substack{T \subseteq B_w \\ \sh(T) = \lambda}}  1
\end{equation}
Furthermore, since for the row shape $(k) \vdash k$, we can count the tilings of all shapes, which gives
\begin{equation} \label{kshape}
  \zeta^{(k)} = \xi^{(k)} = \sum_{\mu \vdash k} \eta^\mu  
\end{equation}
These are some identities that will come handy in later to prove our main result.

    \subsection{Doubilet's inversion formula}


Recall the Young's rule which states that 
\begin{align*}
    M^\lambda &= \bigoplus_{\mu \unlhd \lambda} K_{\mu \lambda} S^\lambda 
\end{align*}
where $K_{\mu\lambda}$ are the Kostka numbers. If we denote by $\chi^\lambda$ the irreducible character of $\mathfrak{S}_n$ corresponding to the Specht module $S^\lambda$, then this implies
\begin{align*}
\zeta^\lambda &= \sum_{\mu \unlhd \lambda} K_{\mu\lambda} \chi^\mu
\end{align*} 
We also know that for partitions $\lambda, \mu \vdash n$, the coefficients $K_{\mu \lambda} \neq 0$ if and only if $\mu \unlhd \lambda$. Furthermore, $K_{\lambda \lambda} = 1$. This implies that the Kostka matrix $K = (K_{\mu \lambda})$ is invertible. Therefore, we can write the irreducible character $\chi^\lambda$ as a linear combination of $\zeta^\lambda$'s. Such an inversion formula was given by Doubilet in \cite{doubilet}, which states
\begin{align*}
\chi^\lambda &= \sum_{\substack{\sigma \in \mathfrak{S}_n  \\ \widetilde{\sigma \lambda} \vdash n }} \sgn(\sigma) \zeta^{\widetilde{\sigma \lambda}}
\end{align*}
where $\sigma \lambda$ is the sequence defined as
\begin{align*}
\sigma \lambda = (\lambda_1 + \sigma(1) - 1, \lambda_2 + \sigma(2)-2, \cdots, \lambda_n + \sigma(n)-n)
\end{align*}
$(\sigma \lambda) _i = \lambda_i + \sigma(i) - i$ for all $i =1, 2, \dots, n$ and $\widetilde{\sigma \lambda}$ is the rearrangement of this sequence in a weakly decreasing order. Since $\lambda_i = 0$ for $\ell(\lambda) < i \leq n$, this implies that for $\widetilde{\sigma \lambda}$ to be a partition, we should have
\begin{align*}
\sigma(i) - i &\geq 0  
\end{align*}
equivalently $\sigma(i) \geq i$ for all $ \ell(\lambda) < i \leq n$. This implies $\sigma(n)=n$ and since $\sigma$ is a bijection and hence injective $\sigma(n-1) = n-1$. Inductively, we have $\sigma(i) = i$ for $\ell(\mu) < i < n$. Let $\overline{\mathfrak{S}}_\ell(\lambda)$ be the subgroup of permutations that pointwise fix all the elements from $\ell(\lambda)+1$ to $n$. Then, the above observation implies that the sum on the right of the inversion formula can be taken over $\sigma \in \overline{\mathfrak{S}}_{\ell(\lambda)}$ such that $\widetilde{\sigma \lambda} \vdash n$.  Hence we can restate the inversion formula as:
\begin{lemma}\label{doubiletinverse}
(Doubilet's inversion formula) Let $\lambda \vdash n$, then keeping the notation of this section, we have
\begin{align*}
\chi^\lambda &= \sum_{\substack{\sigma \in \overline{\mathfrak{S}}_{\ell(\lambda)}  \\ \widetilde{\sigma \lambda} \vdash n }} \sgn(\sigma) \zeta^{\widetilde{\sigma \lambda}}
\end{align*}
\end{lemma}
For our purposes later, we will be able to reduce this sum to a smaller indexing set using the condition that $\sigma \lambda$ needs to be a composition of $n$. This enforces that none of the parts of $\sigma \lambda$ should be less than $1$. For the time being, we show a simple example.

\begin{example} \label{MSM}
For a positive integer $n \geq 1$, take the partition $\lambda = (n-1,1) \vdash n$, then
\begin{align*}
    \chi^{(n-1,n)} &= \zeta^{(n-1, 1)}  - \zeta^{(n)}
\end{align*}
which reminiscent of the fact that the $M^{(n-1,1)} = S^{(n-1, 1)} \oplus M^{(n)}$, where we know that $S^{n-1,n}$ is the regular representation and $M^{(n)}$ is the trivial representation of $\mathfrak{S}_n$.
\end{example}
    \subsection{Face numbers of permutohedron}    

The standard permutohedron is an example of a convex polytope associated to permutations. To each permutation $w \in \mathfrak{S}_n$, we associate a point in $\mathbb{R}^n$
\begin{align*}
    p_w &= (w(1), w(2), \cdots, w(n)) \in \mathbb{R}^n
\end{align*}
The \emph{standard permutohedron} $\Pi_n$ is defined to be the convex hull of these points, i.e.
\begin{align*}
    \Pi_n &= \conv \{ p_w: w \in \mathfrak{S}_n\}
\end{align*}
Note that for each $p_w$, the sum of all the coordinates equals $1 + 2 + \cdots  + n$, i.e.
\begin{align*}
    \sum_{i=1}^n p_{w, i} = 1 + 2 + \cdots + n  = \binom{n+1}{2}
\end{align*}
This means that $\Pi_n$ lies in a hyperplane in $\mathbb{R}^n$ and consequently $\dim(\Pi_n) \leq n-1$. Furthermore, for each $p_w$, sum of any $k$ coordinates is atleast $1 + 2 + \dots + k$, i.e. for any $I \subset [n]$
\begin{align*}
    \sum_{i \in I} p_{w, i} \geq 1 + 2 + \cdots + |I| = \binom{|I|+1}{2}
\end{align*}
It is known classically that these inequalities are enough to describe $\Pi_n$. We need the following known result (see for example \cite{ziegler}) from polytope theory. 
\begin{theorem}
The standard permutohedron $\Pi_n$ is a simple $(n-1)$-dimensional convex polytope given by
\begin{align*}
    \Pi_n &= \left\{ \mathbf{x} \in \mathbb{R}^n: \sum_{i=1}^n x_i = n~~~\text{and}~~~\sum_{i \in I} x_i \geq \binom{|I|+1}{2}, ~~\text{for all}~~ I \subset [n] \right\}
\end{align*} 
with face numbers given by
\begin{align*}
    f_k(\Pi_n) &= (n-k)! \begin{Bmatrix} n\\ n-k \end{Bmatrix}
\end{align*}
 for $k =0, \cdots, n-1$, where $\begin{Bmatrix} n\\ k \end{Bmatrix}$ is the Stirling number of the second kind (that counts the partitions of $[n]$ into $k$ blocks).
\end{theorem}
Since $\Pi_n$ is simple convex polytope, its dual $\Pi_n^*$ is a simplicial convex polytope. We can consider the simplicial complex $\Delta_{n-1} := \partial \Pi_n^*$ (the boundary complex of $\Pi_n^*$). Its face numbers are given by
\begin{align*}
    f_{k-1}(\Delta_{n-1}) &= (k+1)! \begin{Bmatrix} n\\ k+1 \end{Bmatrix}
\end{align*}
for $k=0,1, \cdots, n-1$. This implies that the reduced Euler characteristic of $\Delta_{n-1}$ is given by
\begin{align*}
    \sum_{k=0}^{n-1} (-1)^{k-1} f_{k-1}(\Delta_{n-1}) &= \sum_{k=0}^{n-1} (-1)^{k-1} (k+1)! \begin{Bmatrix} n\\ k+1 \end{Bmatrix}\\
    &= \sum_{k=1}^{n} (-1)^{k-1} k! \begin{Bmatrix} n\\ k \end{Bmatrix}
\end{align*}
Since $\Delta_{n-1}$ is homeomorphic to the $(n-2)$-dimensional sphere $S^{n-2}$, its reduced Euler characteristic equals to $(-1)^{n-2}$. This can be rewritten as
\begin{equation} \label{altperm}
    \sum_{k=1}^{n} (-1)^{k} k! \begin{Bmatrix} n\\ k \end{Bmatrix} = (-1)^n
\end{equation}
We will use this identity in later section.

\section{Combinatorics of tiling class functions}
    Characters of the symmetric group $\mathfrak{S}_n$ are examples of class functions. What that means is that they are constant over conjugacy classes in $\mathfrak{S}_n$. This implies that they only depend on cycle structure of evaluating permutation. Simple examples of class functions are $c_i$'s which count the number of cycles of length $i$ in a permutation. It is a result of Frobenius \cite{frobenius} that the every irreducible character of $\mathfrak{S}_n$ is a polynomial function of $c_i$'s, called \emph{character polynomials}. They were also studied later by Murnaghan \cite{murnaghan} and Specht \cite{specht}. Macdonald mentions them in \cite{macdonald} and attribute them to Frobenius. Garsia and Goupil \cite{garsiagoupil} gave an umbral construction of these polynomials. Kerber also studied them in his book \cite{kerberbook} on group actions. Recently, they have reoccurred in the context of representation stability \cite{CEF}. We study our tiling class functions as polynomials in $c_i$'s. The main result in this section is identity \ref{hook3}, which equates two alternating sums of class functions. We provide two proofs of it: one using homology on poset of brick tilings, and the other using reduced Euler characteristic of boundary complex of dual polytope to permutohedron.
    \subsection{Character polynomials}
        
The class functions $\{ c_i \}_{i \in \mathbb{N}}$ are defined as
\begin{align*}
c_i : \mathfrak{S}_n &\longrightarrow \mathbb{N} \\
w &\longmapsto ~\text{number of cycles of $w$ of length $i$}
\end{align*}
Let $\mathbb{Q}[c_1, c_2, \cdots]$ be the ring of polynomials in $c_i$'s with rational coefficient. We call a polynomial $q(c_1, c_2, \cdots) \in \mathbb{Q}[c_1, c_2, \cdots]$ a \emph{class polynomial}. Frobenius \cite{frobenius} showed the following
\begin{theorem}
Let $\lambda = (\lambda_1, \cdots, \lambda_\ell) \vdash n$ and let $w \in \mathfrak{S}_n$ be a permutation of cycle type $ \cyc(w) = \mu = (\mu_1, \cdots, \mu_r) \vdash n$, then $\chi^\lambda(w)$ equals the coefficient of $x_1^{\lambda_1 + \ell -1} x_2^{\lambda + \ell -2} \cdots x_\ell^{\lambda_\ell}$ in the expansion of
\begin{align*}
   \prod_{1 \leq i < j \leq \ell} (x_i - x_j) \prod_{i=1}^r (x_1^{\mu_i} + x_2^{\mu_i} + \cdots + x_\ell^{\mu_i})
\end{align*}
\end{theorem}
The character polynomial $q_\lambda(c_1, c_2, \cdots)$ is defined as the \emph{unique} polynomial in $\mathbb{Q}[c_1, c_2, \cdots ]$ such that for all partitions $\lambda \vdash k$ and $n \geq \lambda_1 + k$, we have
\begin{equation}
    \chi^{\lambda[n]}(w) = q_\lambda(c_1(w), c_2(w), \cdots)
\end{equation}
for all $w \in \mathfrak{S}_n$. Note that the character polynomial $q_\lambda$ as a polynomial does not depend on $n$, and hence gives a uniform description of irreducible characters corresponding to augmented partition $\lambda[n]$ for all symmetric groups at the same time (with $n \geq \lambda_1 + k$). This is quite remarkable.

\begin{example}
The character polynomial for empty partition is the constant function $1$, as it corresponds to the trivial representation. The character polynomial for$\lambda = 1$ is given by
\begin{align*}
    q_{(1)} &= c_1 - 1
\end{align*}
as $\lambda[n] = (n-1,1)$ corresponds to the regular representation.
\end{example}

For $\lambda \vdash k$, there is a special family of class polynomials called \emph{binomial class polynomials} $\binom{c}{\lambda}$ defined as 
\begin{align*}
\binom{c}{\lambda } &:=  \prod_{i=1}^{n} \binom{c_i}{m_i(\lambda)} 
\end{align*}
where $m_i(\lambda)$ is the multiplicity of $i$ in $\lambda$. As an example, for $\lambda = (4, 2, 2, 1, 1)$, we have
\begin{align*}
\binom{c}{\lambda } &:=   \binom{c_4}{1} \binom{c_2}{2} \binom{c_1}{2} 
\end{align*}

Notice that for $\lambda \vdash k$ and $w \in \mathfrak{S}_n$ with $n \geq k$, having a unordered crackless tiling of $\lambda$ by $\pi$ is equivalent to choosing $m_i(\lambda)$ bricks from all $c_i(w)$ bricks of $w$ of length $i$, for each $i=1, \cdots, n$. This implies
\begin{equation}
\eta^\lambda = \binom{c}{\lambda}
\end{equation}
as class functions on $\mathfrak{S}_n$. This implies the generating function equality
\begin{equation} \label{gencrackless}
\sum_{k=0}^\infty \left( \sum_{\mu \vdash k} \eta^\mu \right) t^k = \prod_{i=1}^\infty(1+t^i)^{c_i}
\end{equation}
in the ring of formal power series $\mathbb{Q}[c_1,c_2 \cdots ][[t]]$.

    \subsection{Characters corresponding to two row partitions}
        We consider the case of $\lambda = (k)$ to generalize example \ref{MSM}. In this case, we are looking at the augmented partition $\lambda[n] = (n-k, k) \vdash n$ for $n \geq 2k$. The
 \ref{doubiletinverse} implies as class functions on $\mathfrak{S}_n$
\begin{align*}
\chi^{\lambda[n]} &= \sum_{\substack{\sigma \in \mathfrak{S}_2 }} \sgn(\sigma) \zeta^{\widetilde{\sigma \lambda[n]}} \\
&=  \zeta^{(n-k,k)} - \zeta^{(n-k+1,k-1)} \\
&=  \zeta^{(k)} - \zeta^{(k-1)} \\
&= \sum_{\mu \vdash k} \eta^ \mu - \sum_{\mu \vdash k-1} \eta^ \mu
\end{align*}
Therefore from equation \ref{gencrackless} we have the following identity in $\mathbb{Q}[c_1, c_2, \cdots][[t]]$
\begin{equation} \label{tworow}
    \sum_{k=0}^\infty q_{(k)}t^{k} = (1-t) \prod_{i=1}^\infty (1-t^i)^{c_i}
\end{equation}
We should be careful about stating the above equality. Becuase what we really mean is that the coefficients of $t^n$ are equal as polynomial functions of $c_i$'s.
\begin{example}
We find first few character polynomials from the above equality. A table of them can also be found\cite{kerberbook}
\begin{align*}
    q_{()} &= 1 \\
    q_{(1)} &= \binom{c_1}{1} - 1 \\
    q_{(2)} &=  \binom{c_2}{1} + \binom{c_1}{2} - \binom{c_1}{1} \\
    q_{(3)} &= \binom{c_3}{1} + \binom{c_1}{1}\binom{c_2}{1} + \binom{c_1}{3} - \binom{c_2}{1} - \binom{c_1}{2}
\end{align*}

\end{example}

    \subsection{Characters corresponding to hook partitions}

Let us consider $\lambda = (1^k) \vdash k$. The augmented partition in this case is the hook partition $\lambda[n] = (n-k, 1^k) \vdash n$, for $n \geq 2k$. Recall that the Doubilet inversion formula \ref{doubiletinverse} says
\begin{equation} \label{hook1}
\chi^{\lambda[n]} = \sum_{\sigma \in \overline{\mathfrak{S}}_{k+1}} \sgn(\sigma) \zeta^{\widetilde{\sigma \lambda[n]}}
\end{equation}
where $\overline{\mathfrak{S}}_{k+1}$ denote the subgroup consisting of permutations in $\mathfrak{S}_{n}$ that fixes all the elements from $k+2$ to $n$. The sum can further be  restricted to those $\sigma \in \overline{\mathfrak{S}}_k$ for which $\widetilde{\sigma \lambda[n]}$ is a partition of $n$, i.e.
\begin{align*}
(\widetilde{\sigma \lambda [n]})_i &= \lambda_i + \sigma(i) - i \geq 0
\end{align*}
For $i = 2, \dots k+1$, $\lambda_i = 1$, therefore we are looking for $\sigma \in \overline{\mathfrak{S}}_{k+1}$ such that $\sigma(i) \geq i-1$. Note that there is no condition on $\sigma(1)$ as $n \geq 2k$. So we can say that
\begin{align*}
\sigma(i) \geq i-1
\end{align*}
for all $i = 1, \dots, k+1$. 
This implies such a permutation is uniquely determined by the set
\begin{align*}
      S(\sigma) = \{ (i, \sigma(i)): \sigma(i) \geq i\}
\end{align*}
Fixing $j = k+1- \sigma(1)$, we see that there are $2^{j-1}$ such permutations and each such permutation gives a composition $\mu \vDash j$. On the other hand, given a composition $\mu \vDash j$, we can construct $\sigma$ by taking
\begin{multline}
      S(\sigma) = \{ (1, k+1-j), (k+1-j+1, k+1-j+\mu_1), \\ (k+1-j+\mu_1+1, k+1-j+\mu_1+\mu_2), \cdots \}
\end{multline}
And for such a $\sigma$, we have
\begin{align*}
    \sigma \lambda[n] &= (n-\sigma(1),  
    \mu_1, \mu_2, \cdots, \mu_{\ell}) \\
    \left<\sigma \lambda[n] \right>&= ( \mu_1, \cdots, \mu_\ell) \vDash j 
\end{align*}
Combining these we can prove the following theorem:

\begin{theorem} \label{hook2}
For $\lambda = (1^k) \vdash k$ and $n \geq 2k$, we have the following equality of class functions
\begin{align*}
    \chi^{\lambda[n]} &= \sum_{j=0}^k (-1)^{j+1} \left(  \sum_{\mu \vDash j} (-1)^{\ell(\mu)} \zeta^\mu \right)
\end{align*}
\end{theorem}
\begin{proof}
Due to the observations above, we can simplify equation \ref{hook1} as follows: Fix $\sigma(1) = k+1-j$, now each permutation gives a unique composition $\mu$ of $j$, so we sum over all compositions of $j$. And we do this over all possible images $\sigma(1)$, which can be any integer from 1 to $k+1$. This translates to $j$ varying from $0$ to $k$.
\begin{align*}
    \chi^{\lambda[n]} &= \sum_{\sigma \in \overline{\mathfrak{S}_{k+1}}} \sgn(\sigma) \zeta^{\widetilde{\sigma \lambda}} \\
     &= \sum_{\sigma \in \overline{\mathfrak{S}_{k+1}}} \sgn(\sigma) \zeta^{\left< \widetilde{\sigma \lambda} \right>} \\
 &= \sum_{j=0}^k (-1)^{k+1 - j} \left( \sum_{\mu \vDash j} (-1)^{\sum \mu +1} \zeta^\mu \right) \\
 &= \sum_{j=0}^k (-1)^{k+1 - j} \left( \sum_{\mu \vDash j} (-1)^{\ell(\mu) + j} \zeta^\mu \right) \\
 &= (-1)^{k+1 }  \sum_{j=0}^k \left( \sum_{\mu \vDash j} (-1)^{\ell(\mu)} \zeta^\mu \right)
\end{align*}
In above, we appropriately kept track of the sign as the cycle type of $\sigma$ with $\sigma(1) = k+1-j$ that give rise to composition $\mu$ is given by the weakly decreasing rearrangement of $(k+1-j, \mu_1+1,  \cdots, \mu_\ell +1)$. This gives
\begin{align*}
    \sgn(\sigma) &= \bigg(\sum_{i=1}^{\ell(\mu)} (\mu_i + 1) \bigg)  + k +1 - j \\
    &= \ell(\mu) + k +1 
\end{align*}
This explains the simplification in the sum above.
\end{proof}



Now we move to an interesting observation that simplifies further the calculation of $\chi^{\lambda[n]}$ for $\lambda = (1^k) \vdash k$ where $n \geq 2k$. This result should be thought of as the main result of this paper.

\begin{theorem} \label{hook3}
For a positive integer $j \leq n$, we have the following equality of class functions
\begin{align*}
 \sum_{\mu \vDash j} (-1)^{\ell(\mu)} \zeta^\mu &= \sum_{\mu \vdash j}  (-1)^{\ell(\mu)} \eta^\mu
\end{align*}
on $\mathfrak{S}_n$
\end{theorem}
The left-hand side is an alternating sum over compositions of $j$, while the right-hand side is an alternating sum over partitions of $j$. It is worth noticing that not only the sum on the right has fewer terms but even each term is \emph{smaller}. In the sections to follow, we provide two proofs of this result.

    \subsection{The Tiling poset}

Fix a positive integer $j \leq n$ and $w \in \mathfrak{S}_n$, then we can consider the set $\Til(w;j)$ of all ordered brick tilings of all compositions of $k$ by $w$, that is
\begin{align*}
\Til(w;j) &:= \bigcup \left\{ \widetilde{B}_w(\lambda) : \lambda \vDash j \right\}
\end{align*}
We equip this set with a partial order that is induced by the following covering relations: For two ordered brick tilings $A' = (A'_1, \dots, A'_{k+1})$ and $A = (A_1, \dots, A_k)$ in $\Til(w;j)$, we say $A$ covers $A'$, and write $A' \lessdot A$ if both of the following conditions are satisfied:
\begin{itemize}
\item
$\sh(A') \lessdot \sh(A)$ in $\Comp(k)$.
\item
There exists a unique $t \in \{1, \cdots, k\}$ such that for all $i = 1, \dots, k$
\begin{align*}
A_i &= \begin{cases}
A'_i ~~~&\text{for}~i<t \\
A'_i \cup A'_{i+1}~~~&\text{for}~i=t \\
A'_{i+1}~~~&\text{for}~i>t
\end{cases}
\end{align*}
\end{itemize}

\begin{example}
For $w = (31)(4)(5,2)$, and $j=5$, we show the Hasse diagram of the poset $\Til(w;j)$ in figure \ref{wk4}. 
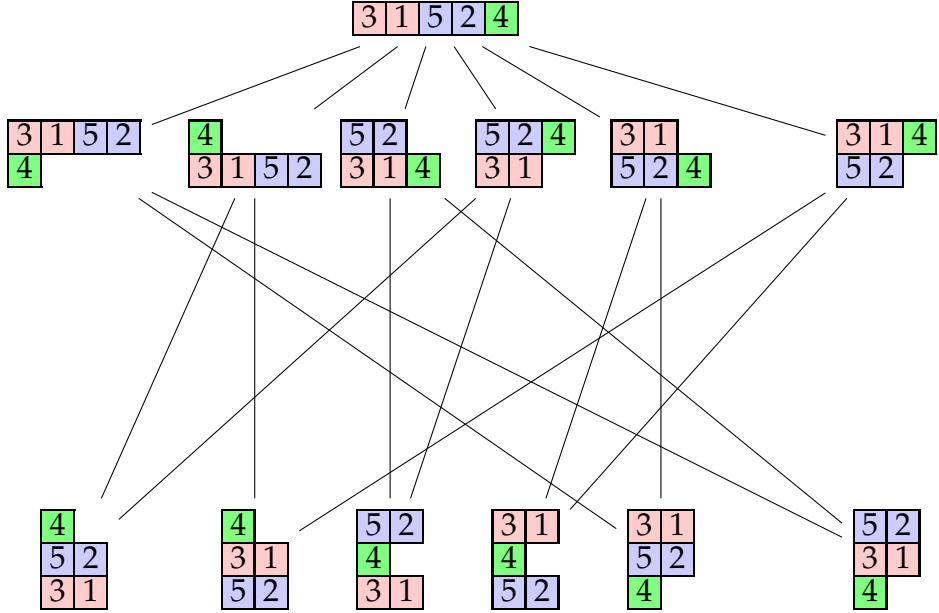
\begin{figure} \label{wk4}
    \centering
    \begin{tikzpicture}[scale=.3]
  \node (m) at (40,0) {\ytableausetup{boxsize=1em}\ytableaushort{52,31,4}*[*(red!20)]{0,2}*[*(blue!20)]{2}*[*(green!50)]{0,0,1}};
  \node (l) at (30,0) {\ytableausetup{boxsize=1em}\ytableaushort{31,52,4}*[*(red!20)]{2}*[*(blue!20)]{0,2}*[*(green!50)]{0,0,1}};
  \node (k) at (24,0)  {\ytableausetup{boxsize=1em}\ytableaushort{31,4,52}*[*(red!20)]{2}*[*(green!50)]{0,1}*[*(blue!20)]{0,0,2}} ;
  \node (j) at (18,0)  {\ytableausetup{boxsize=1em}\ytableaushort{52,4,31}*[*(red!20)]{0,0,2}*[*(blue!20)]{2}*[*(green!50)]{0,1}};
  \node (i) at (12,0)  {\ytableausetup{boxsize=1em}\ytableaushort{4,31,52}*[*(red!20)]{0,2}*[*(blue!20)]{0,0,2}*[*(green!50)]{1}};
  \node (h) at (4,0)  {\ytableausetup{boxsize=1em}\ytableaushort{4,52,31}*[*(red!20)]{0,0,2}*[*(blue!20)]{0,2}*[*(green!50)]{1}} ;
  \node (g) at (40,18) {\ytableausetup{boxsize=1em}\ytableaushort{314,52}*[*(red!20)]{2}*[*(blue!20)]{0,2}*[*(green!50)]{3}};
  \node (f) at (30,18) {\ytableausetup{boxsize=1em}\ytableaushort{31,524}*[*(red!20)]{2}*[*(blue!20)]{0,2}*[*(green!50)]{0,3}} ;
  \node (e) at (24,18)  {\ytableausetup{boxsize=1em}\ytableaushort{524,31}*[*(red!20)]{0,2}*[*(blue!20)]{2}*[*(green!50)]{3}} ;
  \node (d) at (18,18)  {\ytableausetup{boxsize=1em}\ytableaushort{52, 314}*[*(red!20)]{0,2}*[*(blue!20)]{2,0}*[*(green!50)]{0,3}};
  \node (c) at (12,18)   {\ytableausetup{boxsize=1em}\ytableaushort{4,3152}*[*(red!20)]{0,2}*[*(blue!20)]{0,4}*[*(green!50)]{1}};
  \node (b) at (4,18)  {\ytableausetup{boxsize=1em}\ytableaushort{3152,4}*[*(red!20)]{2}*[*(blue!20)]{4}*[*(green!50)]{0,1}};
  \node (a) at (20,24) {\ytableausetup{boxsize=1em}\ytableaushort{31524}*[*(red!20)]{2}*[*(blue!20)]{4}*[*(green!50)]{5}};
   \draw (a)  -- (b); \draw  (a) -- (c) ;\draw  (a) -- (d) ; \draw  (a) -- (e); \draw  (a) -- (f); \draw  (a) -- (g);
  \draw (b)  -- (l); \draw  (b) -- (m) ;
  \draw (c)  -- (h); \draw (c)  -- (i); 
  \draw (d)  -- (j); \draw (d)  -- (m); 
 \draw  (e) -- (h) ;\draw  (e) -- (j) ;
\draw  (f) -- (k) ;\draw  (f) -- (l) ;
 \draw (g)  -- (i); \draw (g)  -- (k);
\ytableausetup{boxsize=1em}
\end{tikzpicture}
    \caption{Poset $\Til(w;j)$ for $w = (3,1)(4)(5,2)$ and $j =4$}
    \label{tilingposet}
\end{figure}
\end{example}

Notice that in general $\Til(w;j)$ can be thought of as disjoint union of smaller posets, that is,
\begin{equation} \label{partition}
    \Til(w; j) = \bigsqcup_{ \substack{T \subseteq B_w} }  \Til(T; j)
\end{equation}
where $\Til(T;j)$ is the set of all ordered brick tilings of all compositions of $j$, whose brick set is $T$. From the poset $\Til(w; j)$, we define following chain complex over field with 2 elements $\mathbb{F}_2$
\begin{align*}
0 \longrightarrow C_{j-1} \longrightarrow C_{j-2} \longrightarrow \cdots \longrightarrow C_{1} \longrightarrow C_{0} \longrightarrow 0
\end{align*}
where the chains are defined as
\begin{align*}
C_i = \mathbb{F}_2\left\{ e_A:  A \in \Til(w;j)~\text{such that}~ \sh(A) \vDash_{j-i} j\right\}
\end{align*}
where $e_T$ are just formal symbols. And the differentials is induced by the down map of the poset, i.e. for $i=1, \dots, j-1$ 
\begin{align*}
\partial_i : C_i &\longrightarrow C_{i-1} \\
e_A &\longmapsto\sum_{A' \lessdot T}  e_{A'}
\end{align*}
To show we indeed have a chain complex we still need to show that $\partial^2 = 0$. Without loss of generality, assume that there exist some $A''< A$ such that $[A'', A]$ is an interval of length 2 . Then let
\begin{align*}
\partial_{i-1} \circ \partial_{i} (e_A) = \sum \alpha_{AA''}e_{A''}
\end{align*}
where the  sum is over all $A''$ for which there exists $A'$ such that $A'' \lessdot A' \lessdot A$ and the coefficients are given by 
\begin{align*}
\alpha_{AA''} &= |\{ A': \text{such that}~A''\lessdot A' \lessdot A\}| \pmod{2}
\end{align*}
Let $\beta_{AA''} = |\{ A': \text{such that}~A''\lessdot A' \lessdot A\}|$. We claim that for every pair $(A'',A)$ such that there exists $A'$ with $A'' \lessdot A' \lessdot A$, we have $\beta_{AA''} = 2$. Let $A''= (A''_1, A''_2, \dots, A''_{k+2})$ and $A =  (A_1, A_2, \dots, A_{k})$, then we can have one of the following situations:
\begin{itemize}

\item
We have two indices $a, b$, with $a<b$, such that
\begin{align*}
A_i &= 
\begin{cases}
A''_i ~~~&\text{if}~i < a \\
A''_i \cup A''_{i+1}~~~&\text{if}~ i =a \\
A''_{i+1}~~~&\text{if}~ a \leq i \leq b \\
A''_{i+1} \cup A''_{i+2}~~~&\text{if}~i = b \\
A''_{i+2}~~~&\text{if}~b < i \leq k
\end{cases}
\end{align*}
In this case, there are two possible $A'$: $(A''_1, \dots, A''_a \cup A''_{a+1}, \dots, A''_{k+2})$ or  $(A''_1, \dots, A''_b \cup A''_{b+1}, \dots, A''_{k+2})$. Then, $\beta_{AA''} = 2$.

\item
We have an index $a$ such that
\begin{align*}
A_i &= 
\begin{cases}
A''_i ~~~&\text{if}~i < a \\
A''_i \cup A''_{i+1} \cup  A''_{i+2}~~~&\text{if}~ i =a \\
A''_{i+2}~~~&\text{if}~a < i\leq k
\end{cases}
\end{align*}
In this case, $A'= (A''_1, \dots, A''_a \cup A''_{a+1}, \dots, A''_{k+2})$ or  $A'=(A''_1, \dots, A''_{a+1} \cup A''_{a+2}, \dots, A''_{k+2})$, and we also have that $\beta_{AA''} = 2$.
\end{itemize}
Since we are dealing with coefficients $\mathrm{mod}~ 2$, in both cases $\alpha_{AA''} = 0$. This proves that $\partial^2 = 0$, and so we indeed have a chain complex. To illustrate the idea, we show an example


\begin{example}
For $w = (3,1)(5, 2)(4)$, we get the following for $k = 5$
\begin{align*}
\partial ~\ytableausetup{boxsize=1em}\ytableaushort{31524}*[*(red!20)]{2}*[*(blue!20)]{4}*[*(green!50)]{5} &= {\ytableausetup{boxsize=1em}\ytableaushort{31,524}*[*(red!20)]{2}*[*(blue!20)]{0,2}*[*(green!50)]{0,3}} +
{\ytableausetup{boxsize=1em}\ytableaushort{524,31}*[*(red!20)]{0,2}*[*(blue!20)]{2}*[*(green!50)]{3}} + {\ytableausetup{boxsize=1em}\ytableaushort{3152,4}*[*(red!20)]{2}*[*(blue!20)]{4}*[*(green!50)]{0,1}} +
{\ytableausetup{boxsize=1em}\ytableaushort{4,3152}*[*(red!20)]{0,2}*[*(blue!20)]{0,4}*[*(green!50)]{1}} + {\ytableausetup{boxsize=1em}\ytableaushort{314,52}*[*(red!20)]{2}*[*(blue!20)]{0,2}*[*(green!50)]{3}} +  {\ytableausetup{boxsize=1em}\ytableaushort{52, 314}*[*(red!20)]{0,2}*[*(blue!20)]{2,0}*[*(green!50)]{0,3}} \\
\partial~ {\ytableausetup{boxsize=1em}\ytableaushort{31,524}*[*(red!20)]{2}*[*(blue!20)]{0,2}*[*(green!50)]{0,3}} &= {\ytableausetup{boxsize=1em}\ytableaushort{31,52,4}*[*(red!20)]{2}*[*(blue!20)]{0,2}*[*(green!50)]{0,0,1}} + {\ytableausetup{boxsize=1em}\ytableaushort{31,4,52}*[*(red!20)]{2}*[*(green!50)]{0,1}*[*(blue!20)]{0,0,2}} \\
\partial~{\ytableausetup{boxsize=1em}\ytableaushort{524,31}*[*(red!20)]{0,2}*[*(blue!20)]{2}*[*(green!50)]{3}}  &=  {\ytableausetup{boxsize=1em}\ytableaushort{52,4,31}*[*(red!20)]{0,0,2}*[*(blue!20)]{2}*[*(green!50)]{0,1}} + {\ytableausetup{boxsize=1em}\ytableaushort{4,52,31}*[*(red!20)]{0,0,2}*[*(blue!20)]{0,2}*[*(green!50)]{1}} \\
\partial~{\ytableausetup{boxsize=1em}\ytableaushort{3152,4}*[*(red!20)]{2}*[*(blue!20)]{4}*[*(green!50)]{0,1}} &= 
{\ytableausetup{boxsize=1em}\ytableaushort{31,52,4}*[*(red!20)]{2}*[*(blue!20)]{0,2}*[*(green!50)]{0,0,1}} + {\ytableausetup{boxsize=1em}\ytableaushort{52,31,4}*[*(red!20)]{0,2}*[*(blue!20)]{2}*[*(green!50)]{0,0,1}}\\
\partial~{\ytableausetup{boxsize=1em}\ytableaushort{4,3152}*[*(red!20)]{0,2}*[*(blue!20)]{0,4}*[*(green!50)]{1}} &= 
{\ytableausetup{boxsize=1em}\ytableaushort{4,31,52}*[*(red!20)]{0,2}*[*(blue!20)]{0,0,2}*[*(green!50)]{1}} +  {\ytableausetup{boxsize=1em}\ytableaushort{4,52,31}*[*(red!20)]{0,0,2}*[*(blue!20)]{0,2}*[*(green!50)]{1}} \\
\partial~{\ytableausetup{boxsize=1em}\ytableaushort{314,52}*[*(red!20)]{2}*[*(blue!20)]{0,2}*[*(green!50)]{3}} &= 
{\ytableausetup{boxsize=1em}\ytableaushort{31,4,52}*[*(red!20)]{2}*[*(blue!20)]{0,0,2}*[*(green!50)]{0,1}} + {\ytableausetup{boxsize=1em}\ytableaushort{4,31,52}*[*(red!20)]{0,2}*[*(blue!20)]{0,0,2}*[*(green!50)]{1}} \\
\partial~{\ytableausetup{boxsize=1em}\ytableaushort{52, 314}*[*(red!20)]{0,2}*[*(blue!20)]{2,0}*[*(green!50)]{0,3}}  &= 
{\ytableausetup{boxsize=1em}\ytableaushort{52,31,4}*[*(red!20)]{0,2}*[*(blue!20)]{2}*[*(green!50)]{0,0,1}} +  {\ytableausetup{boxsize=1em}\ytableaushort{52,4,31}*[*(red!20)]{0,0,2}*[*(blue!20)]{2}*[*(green!50)]{0,1}} \\
\partial~{\ytableausetup{boxsize=1em}\ytableaushort{31,52,4}*[*(red!20)]{2}*[*(blue!20)]{0,2}*[*(green!50)]{0,0,1}} &= \partial~{\ytableausetup{boxsize=1em}\ytableaushort{31,4,52}*[*(red!20)]{2}*[*(green!50)]{0,1}*[*(blue!20)]{0,0,2}} = \partial~ {\ytableausetup{boxsize=1em}\ytableaushort{52,31,4}*[*(red!20)]{0,2}*[*(blue!20)]{2}*[*(green!50)]{0,0,1}} = \partial~ {\ytableausetup{boxsize=1em}\ytableaushort{52,4,31}*[*(red!20)]{0,0,2}*[*(blue!20)]{2}*[*(green!50)]{0,1}} = \partial~ {\ytableausetup{boxsize=1em}\ytableaushort{4,31,52}*[*(red!20)]{0,2}*[*(blue!20)]{0,0,2}*[*(green!50)]{1}} = \partial~ {\ytableausetup{boxsize=1em}\ytableaushort{4,52,31}*[*(red!20)]{0,0,2}*[*(blue!20)]{0,2}*[*(green!50)]{1}} = 0
\end{align*}
\end{example}
\noindent An element $J \in C_i$ for some positive integer $i < j$, can be written as
\begin{align*}
    J &= \sum_{T \in L} e_T 
\end{align*}
The set $L$ is called the \emph{support} of $J$, and we denote it by $\supp(J)$. For a subspace $V$ of the chain complex, let $\overline{V}$ be the subspace of $V$ generated by elements in $V$ whose support is crackless brick tilings, and let $\widetilde{V}$ be the subspace of $V$ generated by elements in $V$ whose support is cracked brick tilings. An interesting consequence of equation \ref{partition} is the following decomposition:
\begin{align*}
    \ker \partial_i &= \widetilde{\ker \partial_i} \oplus \overline{\ker \partial_{i}}~~~~\text{and}~~~~\im \partial_{i+1} = \widetilde{\im \partial_{i+1}} \oplus \overline{\im \partial_{i+1}}
\end{align*}
One inclusion is obvious. The other inclusion is the consequence of the fact that a cracked brick tiling and a crackless brick tiling with $j-i$ parts cannot have the same set of bricks for all $i=1, \cdots, j-1$. Therefore the decomposition follows from the partition \ref{partition}. Hence, we can write the homology of the above chain complex as
\begin{align*}
H_i(C) &= \frac{\text{ker} \partial_{i}}{ \text{im} \partial_{i+1}} \\
&= \frac{\widetilde{\text{ker} \partial_{i}} \oplus \overline{\text{ker} \partial_{i}}}{ \widetilde{\text{im} \partial_{i+1}} \oplus  \overline{\text{im} \partial_{i+1}}} 
\end{align*}
Notice in proving $\partial^2 = 0$, we concluded something more than that. We proved that each subinterval of rank 2 in $\Til(w;j)$ is isomorphic to the boolean interval of length $B(2)$ as posets, which implies $\widetilde{\text{im} \partial_{i+1}} = \widetilde{\text{ker} \partial_{i}}$.  This simplifies the above expression to
\begin{align*}
H_i(C)  &= \frac{\overline{\text{ker} \partial_{i}}}{ \overline{\text{im} \partial_{i+1}}} 
\end{align*}
Now notice that every crackless brick tiling with $k-i$ parts  is in the kernel of $\partial_i$, hence
\begin{align*}
H_i(C)&= \frac{\overline{C}_i}{ \overline{\text{im} \partial_{i+1}}}
\end{align*}
We know that $\im \partial_{i+1}$ is generated by $\partial_i e_A$ for $e_A \in C_{i+1}$. The crackless part of $\im  \partial_{i+1}$ is generated by image of $e_A \in C_{i+1}$ which have exactly one crack, and for each such $A$ we have
\begin{align*}
\partial_{i+1} e_A &= e_{B} + e_{B'}
\end{align*}
where $B$ and $B'$ you get by ``\emph{cracking}'' $A$ in either way.
This implies we can thought of $H_i(C)$ as
\begin{align*}
H_i(C) &= \frac{\overline{C}_i}{(e_{B} = e_{B'})} 
\end{align*}
Since this says $e_{B} = e_{B'}$, whenever the Young diagram of $B$ differs from $B'$ by a simple transposition. Now notice that if $e_A \in \overline{C}_i$, then for every permutation $A'$ of $A$, $e_{A'} \in \overline{C}_i$. But the relation $e_{B} = e_{B'}$ identify all of them since all permutations are generated by transpositions. This implies $H_i(C)$ is generated by $e_A$, where $A$ is an unordered crackless brick tiling with $j-i$ parts. Hence
\begin{equation} \label{homology}
\dim H_i(C) = \sum_{\substack{\mu \vdash j \\ \ell(\mu) = j-i}} \eta^\mu(w)
\end{equation}
Now using this, we provide a proof for theorem \ref{hook3}.

\begin{proof}
Given positive integer $j \leq n$, for any permutation $w \in \mathfrak{S}_n$, we construct the poset $\Til(w;j)$, and consider the chain complex associated to it as defined above. The construction implies that the $\dim(C_i)$ is given by
\begin{align*}
    \dim(C_i) &= \sum_{\substack{\mu \vDash j \\ \ell(\mu) = j-i}} \zeta^\mu(w)
\end{align*}
Now we can compute Euler characteristic of this chain complex as an alternating sum of these dimensions, which gives
\begin{align*}
    \sum_{i = 0}^{j-1} (-1)^i\dim(C_i) &= \sum_{i=0}^{j-1} (-1)^{i} \sum_{\substack{ \mu \vDash j  \\ \ell(\mu) = j-i} } \zeta^\mu (w) \\
    &= \sum_{\mu \vDash j} (-1)^{j-\ell(\mu)} \zeta^\mu (w)
\end{align*}
while, computing the Euler characteristic via alternating sum of dimensions of homologies (employing equation \ref{homology} ) gives
\begin{align*}
    \sum_{i=0}^{j-1} (-1)^i\dim(H_i) &=  \sum_{i=0}^{j-1} (-1)^{i} \sum_{\substack{ \mu \vdash j  \\ \ell(\mu) = j-i} } \eta^\mu (w) \\
    &= \sum_{\mu \vdash j} (-1)^{j-\ell(\mu)} \eta^\mu (w)
\end{align*}
Since both of the them are Euler characteristic of the same complex, hence they are equal for each $w \in \mathfrak{S}_n$. This implies that the equality holds as equality of class functions on $\mathfrak{S}_n$.
\end{proof}

    \subsection{Counting proof}
        Though the proof using homology on brick tiling poset was quite interesting, we were also tempted to provide a counting proof. This is given below:

\begin{proof}
Since $\zeta^\lambda = \zeta^\mu$ for $\mu \vdash j$ and $\lambda \vDash j$ such that $\tilde{\lambda} = \mu$, we can write
\begin{align*}
 \sum_{\mu \vDash j} (-1)^{\ell(\mu) } \zeta^\mu &=  \sum_{\mu \vdash j}  (-1)^{\ell(\mu)}   \sum_{\tilde{\lambda} = \mu} \zeta^\lambda
\end{align*}
The sum $\sum_{\tilde{\lambda} = \mu} \zeta^\lambda(w)$  counts the number of ordered tilings by $w$ of all compositions $\lambda$ whose underlying partition is $\mu$. This can also be regarded as counting all permutations (of parts) of unordered brick tilings of $\mu$. Therefore, we can rewrite
\begin{align*}
\sum_{\mu \vdash j}  (-1)^{\ell(\mu)}   \sum_{\tilde{\lambda} = \mu} \zeta^\lambda(w) &= \sum_{\mu \vdash j}  (-1)^{\ell(\mu)}   \ell(w)!  \xi^\mu(w) 
\end{align*}
where $\xi^\mu(w)$ is the number of unordered tilings of $\mu$ by $w$. For a subset of bricks $T \subseteq B_w$, let $\xi^\mu_T$ be the number of unordered brick tilings of $\mu$ by $T$, then we can count the sum over $T$'s
\begin{align*}
\sum_{\mu \vdash j}  (-1)^{\ell(\mu)}   \ell(\mu)!  \xi^\mu(w)  &= \sum_{\mu \vdash j}  (-1)^{\ell(\mu)}   \ell(\mu)!  \sum_{T \subseteq B_w} \xi^\mu_T \\
&=  \sum_{T \subseteq B_w} \sum_{\mu \vdash j}  (-1)^{\ell(\mu)}   \ell(\mu)!  \xi^\mu_T \\
&=  \sum_{T \subseteq B_w} \sum_{k=1}^j \sum_{\substack{\mu \vdash j \\ \ell(\mu) = k}}  (-1)^{\ell(\mu)}   \ell(\mu)!  \xi^\mu_T\\
&=  \sum_{T \subseteq B_w} \sum_{k=1}^j    (-1)^{k}  k!  \sum_{\substack{\mu \vdash j \\ \ell(\mu) = k}} \xi^\mu_T
\end{align*}
Since the tilings counted in $\sum_{\substack{\mu \vdash j \\ \ell(\mu) = k}} \xi^\mu_T$ is in one to one correspondence with unordered partitions of $T$ into $k$ sets. These are counted by Stirling numbers of second kind $\begin{Bmatrix} \ell(\sh(T))\\ k \end{Bmatrix}$.
\begin{align*}
\sum_{T \subseteq B_w} \sum_{k=1}^j    (-1)^{k}  k!  \sum_{\substack{\mu \vdash j \\ \ell(\mu) = k}} \xi^\mu_T(\pi) &=  \sum_{T \subseteq B_w} \sum_{k=1}^j    (-1)^{k}  k!   \begin{Bmatrix} \ell(\sh(T))\\ k \end{Bmatrix} 
\end{align*}
We have encountered the inner alternating sum in before, and from equation \ref{altperm} it equals $(-1)^{\ell(\sh(T))}$. Hence,
\begin{align*}
 \sum_{\mu \vDash j} (-1)^{\ell(\mu) } \zeta^\mu &=  \sum_{T \subseteq B_w} (-1)^{\ell(\sh(T))} \\
    &=  \sum_{\mu \vdash j} \sum_{\substack{T \subseteq B_w \\  \sh(T) = \mu} }  (-1)^{\ell(\mu)} \\
     &=  \sum_{\mu \vdash j}  (-1)^{\ell(\mu)}  \sum_{\substack{T \subseteq B_w \\  \sh(T) = \mu} } 1
\end{align*}
Now the interior sum just counts the number of unordered tilings of $\mu$ by $B_w$ as indicated by equation \ref{brickwisesh}. Therefore
\begin{align*}
 \sum_{\mu \vDash j} (-1)^{\ell(\mu) } \zeta^\mu  &= \sum_{\mu \vdash j}  (-1)^{\ell(\mu)}  \eta^\mu(w)
\end{align*}

\end{proof}

\noindent Let us illustrate the identity \ref{hook3} using a small example:

\begin{example}
In the table \ref{j=4}, we take the case of $j=4$.
\begin{table}[h]
    \centering
    \begin{tabular}{ccccccccccc}
        $\zeta^{(4)} $ & $=$ & $\eta^{(4)}$ & $+$ & $\eta^{(3,1)}$ & $+$ &  $\eta^{(2,2)}$ & $+$ & $\eta^{(2, 1, 1)}$  & $+$ & $\eta^{(1,1,1,1)}$  \\
        $\zeta^{(3,1)}$ & $=$ &  & & $\eta^{(3,1)}$ &  & & & $2 \eta^{(2, 1, 1)}$ & $+$ & $4 \eta^{(1, 1, 1, 1)}$ \\
        $\zeta^{(1,3)}$ & $=$ &  & & $\eta^{(3,1)}$ &  & & & $2 \eta^{(2, 1, 1)}$ & $+$ & $4 \eta^{(1, 1, 1, 1)}$ \\
        $\zeta^{(2, 2)}$ & $=$ & & & & &  $2\eta^{(2, 2)}$ & $+$ & $2 \eta^{(2,1,1)}$ & $+$ & $12 \eta^{(1,1,1,1)}$ \\
        $\zeta^{(2,1,1)}$ & $=$ & & & & & & & $2 \eta^{(2, 1, 1)}$ & $+$ & $6 \eta^{(1, 1, 1, 1)}$ \\
        $\zeta^{(1,2,1)}$ & $=$ & & & & & & & $2 \eta^{(2, 1, 1)}$ & $+$ & $6 \eta^{(1, 1, 1, 1)}$ \\
        $\zeta^{(1,1,2)}$ & $=$ & & & & & & & $2 \eta^{(2, 1, 1)}$ & $+$ & $6 \eta^{(1, 1, 1, 1)}$ \\
        $\zeta^{(1, 1, 1, 1)}$ & $=$ & & & & & & & &  & $24 \eta^{(1, 1, 1, 1)}$
    \end{tabular}
    \caption{Example for $j=4$}
    \label{j=4}
\end{table}
From the table \ref{j=4}, we can see that $ \sum_{\mu \vDash 4} (-1)^{\ell(\mu) } \zeta^\mu$ equals
\begin{align*}
    - \eta^{(4)} + \eta^{(3,1)}  + \eta^{(2,2)} - \eta^{(2,1,1)} + \eta^{(1,1,1,1)}
\end{align*}

\end{example}

\section{Applications}
    In order to continue our study for generating functions of character polynomals, and apply our results to get some previously known identities we define the notion of stability in first subsection. The generating function relevant to our discussion is cycle-index generating function. We go back to our main identity \ref{hook3} and use it to derive Goupil's generating function identity \cite{goupil} for hook partitions. Lastly, combining this identity with stability of cycle-index generating function, we were able to provide an alternating proof of Rosas' formula \cite{rosas}.

    \subsection{Stability for sequence of polynomials and power series}
    
We say a sequence $(f_n)_{n\in \mathbb{N}}$ of polynomials $f_n \in \mathbb{Q}[x]$, \emph{stabilizes} to $f \in \mathbb{Q}[[x]]$ if for each positive integer $k$, there exists a positive integer $N_k$ such that for all $m, n > N_k$, we can find $f \in \mathbb{Q}[x]$ satisfying
\begin{align*}
[x^k]f_m &= [x^k] f_n = [x^k]f
\end{align*}
where $[x^k]f$ denote the coefficient of $x^k$ in $f$. If $f$ exists, then it is seen to be unique. This definition can also be generalized to multivariable case. We say a sequence $(f_n)_{n \in \mathbb{N}}$ of multivariable polynomials $f_n \in \mathbb{Q}[x_1, x_2, \cdots, x_r]$, \emph{stabilizes} to $f \in \mathbb{Q}[[x_1, \cdots, x_r]]$ if for each sequence of positive integers $k = (k_1, \cdots, k_r)$, there exists a positive integer $N_k$ such that for all $m, n > N_k$, we can find $f \in \mathbb{Q}[x]$ satisfying
\begin{align*}
[x_1^{k_1} \cdots x_r^{k_r}]f_m &= [x_1^{k_1} \cdots x_r^{k_r}] f_n = [x^k]f
\end{align*}
\begin{example}
Let $f_n(x) = x^n + x^{n-1} + \dots + x + 1$ is a sequence which stabilizes to $f(x) = \frac{1}{1-x}$. On the other hand the sequence $f_n(x) = x^n$ does not stabilize. An example for multivariable case would be the sequence
\begin{align*}
    g_n(x, y) &= \sum_{i+j = n} x^i y^j
\end{align*}
this sequence stabilize to
\begin{align*}
    g(x, y) &= \frac{1}{1-x} \frac{1}{1-y}
\end{align*}
\end{example}
For a commutative ring $R$ and an element $f \in  R[[t]]$, we say \emph{coefficients of $f$ stabilize} if there exists a least integer $k$ such that for all $n > k$, we have
\begin{align*}
[t^n] f = [t^{n+1}]f
\end{align*}
The limiting coefficient is called \emph{coefficient of stabilization}, and such a $k$ is called \emph{point of stability}.

\begin{example}
Given a polynomial $p(t) \in \mathbb{Q}[t]$, the coefficents of the power series
\begin{align*}
\frac{p(t)}{1-t} \in \mathbb{Q}[[t]]
\end{align*}
stabilizes with coefficient of stabilization $p(1)$ and point of stability is $\deg(p)+1$. 
\end{example}
We require a generalization of the above notion for the ring $\mathbb{Q}[c_1, c_2, \cdots ][[t]]$ of formal power series. Consider an element $f \in \mathbb{Q}[c_1, c_2, \cdots ][[t]]$, given by
\begin{align*}
    f &= \sum_{n=0}^\infty f_n t^n
\end{align*}
where $f_n \in \mathbb{Q}[c_1, c_2, \cdots]$ are polynomials in $c_i$'s. We say $f$ \emph{stabilizes} if the sequence of polynomials $(f_n)_{n \in \mathbb{N}}$ stabilizes to some $g$. In that case, there exists sequence $(m_1, \cdots, m_r)$ such that we have the equality 
\begin{align*}
    [c_1^{k_1} \cdots c_r^{k_r}] f &=  [c_1^{k_1} \cdots c_r^{k_r}]\frac{g}{1-t}
\end{align*}
for $k_i > m_i$ for all $i$.

\begin{example}
Consider the element $f \in \mathbb{Q}[x][[t]]$ given by
\begin{align*}
    f &= 1 + (1+x)t + (1 + x + x^2)t + \cdots
\end{align*}
then $g = \frac{1}{1-x}$, and therefore $f$ stabilizes to
\begin{align*}
    \frac{1}{1-x} \frac{1}{1-t}
\end{align*}
\end{example}

    \subsection{The cycle-index generating function}
    Recall the \emph{cycle index} of the symmetric group $\mathfrak{S}_n$ which is defined as
\begin{align*}
Z(\mathfrak{S}_n) &= \frac{1}{n!} \sum_{w \in \mathfrak{S}_n} \prod_{i=1}^n x_i^{c_i(w)}
\end{align*}
It is well known, for example from \cite{cameron}, that the generating function $\Gamma$ of cycle indices of $\mathfrak{S}_n$ is given by
\begin{align*}
    \Gamma &= 1 + \sum_{n=1}^\infty Z(\mathfrak{S}_n)t^n
\end{align*}
can also be written as
\begin{equation} \label{cycindgen}
    \Gamma = \exp\left( \sum_{i=1}^\infty \frac{x_i t^i}{i} \right)
\end{equation}
This is a very useful result as indicated by the following example:
\begin{example}
Consider the formal equality
\begin{align*}
   1 + \sum_{n=1}^\infty Z(\mathfrak{S}_n)t^n  &= \exp\left( \sum_{i=1}^\infty \frac{x_i t^i}{i} \right)
\end{align*}
Evaluating the formal partial derivative $\frac{\partial }{ \partial x_k}$ of both sides at $x_i = 1$ for all $i=1,2, \cdots $, gives the following 
\begin{align*}
    \sum_{n=1}^\infty \bigg(  \frac{1}{n!} \sum_{w\in \mathfrak{S}_n} c_k(w)\bigg) t^n &=  \frac{1}{k} \frac{t^k}{1-t}
\end{align*}
which is saying that the expected number of $k$ cycles in a permutation of $\mathfrak{S}_n$ for $n \geq k$ equals $\frac{1}{k}$. This is a classical result in combinatorial probability theory \cite{bona}.
\end{example} 
Some other identities we would like to highlight here are
\begin{equation}
    \Gamma[x_i \longrightarrow x^i] = \frac{1}{1-xt} 
\end{equation}
and
\begin{equation} \label{smallid}
    \Gamma[x_i \longrightarrow 1-x^i] = \frac{1-xt}{1-t}
\end{equation}
where the notation $\Gamma[x_i \longrightarrow f(x_i)]$ means substituting $f(x_i)$ for $x_i$ in the expression for $\Gamma$.
    \subsection{Goupil's generating function identity}
        Going back to our identity \ref{hook3}, we consider the class function
\begin{align*}
    \sum_{\mu \vdash j} (-1)^{\ell(\mu) +1 } \eta^\mu
\end{align*}
from the identity \ref{gencrackless}, we have the following generating function
\begin{align*}
    \sum_{j=0}^\infty \left(\sum_{\mu \vdash j} (-1)^{\ell(\mu) +1 } \eta^\mu \right) t^j &= \prod_{i=1}^\infty \bigg( 1 - (-t)^i\bigg)^{c_i}
\end{align*}
Using this and theorem \ref{hook2} we have another proof of the following identity of Goupil \cite{goupil}  for generating function for hook characters
\begin{theorem}
For $\lambda = (1^k) \vdash k$, we have the following generating function
\begin{equation} \label{alt}
    \sum_{k=0}^\infty q_{(1^k)} t^k \doteq \frac{1}{1+t} \prod_{i=1}^\infty \bigg( 1 - (-t)^i\bigg)^{c_i}
\end{equation}
\end{theorem}
\begin{proof}
Since from theorem \ref{hook2}, we have
\begin{align*}
    \chi^{\lambda[n]} &= \sum_{j=0}^k (-1)^{j+1} \left(  \sum_{\mu \vDash j} (-1)^{\ell(\mu)} \zeta^\mu \right)
\end{align*}
which corresponds to alternating sum of class function whose generating function is given by equation \ref{alt}. Now recall that the generating function for alternating sum of a sequence can be constructed by multiplying the generating function by $\frac{1}{1+t}$. This gives the required result.
\end{proof}

\begin{example}
From the above expression we derive expressions for first few character polynomials. A table of these can also be found in \cite{kerberbook}
\begin{align*}
    q_{()} &= 1 \\
    q_{(1)} &= \binom{c_1}{1} - 1 \\
    q_{(1^2)} &= \binom{c_1}{2} - \binom{c_2}{1} - \binom{c_1}{1} + 1 \\
    q_{(1^3)} &= \binom{c_1}{2} - \binom{c_1}{1} \binom{c_2}{1} + \binom{c_3}{1} - \binom{c_1}{2} + \binom{c_2}{1} + \binom{c_1}{1} - 1
\end{align*}
\end{example}

    \subsection{Rosas' formula for certain Kronecker coefficients}
         Let $\lambda, \mu$ and $\nu$ be partitions of $n$. The Kronecker coefficients $g^\lambda_{\mu \nu }(\mathfrak{S}_n)$ are defined as the coefficient of $\chi^\lambda$ in the expansion of $\chi^\mu\chi^\nu$ into irreducible characters:
\begin{align*}
g^{\lambda}_{\mu\nu}({\mathfrak{S}_n}) = \left< \chi^\lambda, \chi^\mu\chi^\nu\right>_{\mathfrak{S}_n} = \frac{1}{n!} \sum_{\sigma \in \mathfrak{S}_n} \chi^\lambda (\sigma) \chi^\mu(\sigma) \chi^\nu(\sigma)
\end{align*}
Using the equation \ref{alt} for generating function of irreducible hook character, here we will derive the formula for Kronecker coefficients indexed by hooks given by Rosas \cite{rosas}, which says:
\begin{theorem}
Let $\overline{g}_{k_1 k_2 k_3}$ be the reduced Kronecker coefficients corresponding to the triple $((n-k_1, 1^{k_1}), (n-k_2, 1^{k_2}), (n-k_3, 1^{k_3}))$, then
\begin{align*}
\sum_{k_1, k_2, k_3} \overline{g}_{k_1 k_2 k_3} x^{k_1} y^{k_2 } z^{k_3} &= \frac{1 + xyz}{(1-xy)(1-yz)(1-xz)}
\end{align*}
\end{theorem}
\begin{proof}
We start with the substituion $x_i \leftarrow (1-(-x)^i)$ in the cycle index function
\begin{align*}
Z(\mathfrak{S}_n)[x_i \leftarrow (1-(-x)^i)] &= \frac{1}{n!} \sum_{w \in \mathfrak{S}_n} \prod_{i=1}^n  (1-(-x)^i)^{a_i(w)} \\
[x^k] \frac{1}{1+x}Z(\mathfrak{S}_n)[x_i \leftarrow (1-(-x)^i)] &= \frac{1}{n!} \sum_{w \in \mathfrak{S}_n} \chi^{(1^k)[n]}(w)
\end{align*}
This suggests that if we let $Y_n = Z(\mathfrak{S}_n)[x_i \leftarrow (1-(-x)^i)(1-(-y)^i)(1-(-z)^i)]$, we have
\begin{align*}
    [x^{k_1} y^{k_2} z^{k_3}] \frac{1}{(1+x)(1+y)(1+z)}Y_n &= \frac{1}{n!} \sum_{w \in \mathfrak{S}_n} \chi^{(1^{k_1})[n]}(w) \chi^{(1^{k_2})[n]}(w) \chi^{(1^{k_3})[n]}(w) \\
    &= g_{k_1 k_2 k_3} 
\end{align*}
Now notice that $Y := \sum_{n=0}^\infty Y_n t^n$ is given by substituting $[x_i \leftarrow (1-(-x)^i)(1-(-y)^i)(1-(-z)^i)]$ in $\Gamma$ which using equation \ref{cycindgen} becomes
\begin{align*}
    Y &= \exp\left( \sum_{i=1}^\infty \frac{x_i}{i} t^i \right)[x_i \leftarrow (1-(-x)^i)(1-(-y)^i)(1-(-z)^i)]
\end{align*}
Now using equation \ref{smallid}, we have
\begin{align*}
    Y &= \frac{1}{1-t} \frac{(1+xyzt) (1+xt)(1+yt)(1+zt)}{(1-xyt)(1-yzt)(1-xzt)}
\end{align*}
Note that the identity
\begin{align*}
     W := \frac{1}{1-t} \frac{1}{(1-xyt)(1-yzt)(1-xzt)} &= \left( \sum_{i=0}^\infty \left( \sum_{j=0}^n\sum_{a + b + c = j} (xy)^a(yz)^b (xz)^c\right) t^n\right)
\end{align*}
implies that $W$ as member of $\mathbb{Q}[x,y,z][[t]]$ stabilize to
\begin{align*}
    \frac{1}{1-t}  \frac{1}{(1-xy)(1-yz)(1-xz)}
\end{align*}
which in turn implies that $\frac{1}{(1+x)(1+y)(1+z)}Y$ stabilizes to
\begin{align*}
   \frac{1}{1-t} \frac{(1+xyz)}{(1-xy)(1-yz)(1-xz)}
\end{align*}
but $\frac{1}{(1+x)(1+y)(1+z)}Y$ is the generating function of $g_{k_1 k_2 k_3}$. This implies that the generating function for the reduced Kronecker coefficients $\overline{g}_{k_1k_2k_3}$ is given by
\begin{align*}
  \frac{(1+xyz)}{(1-xy)(1-yz)(1-xz)}
\end{align*}
\end{proof}

\bibliographystyle{amsalpha}
\bibliography{main}

\end{document}